\documentclass{amsart}
\usepackage{amsmath,enumitem, bbold,subcaption,mathtools}
\usepackage{hyperref}
\usepackage[nameinlink]{cleveref}

\def\mm{\mathfrak{m}}
\def\nn{\mathfrak{n}}

\def\k{\mathbb{k}}
\def\I{\mathbb{I}}

\def\N{\mathbb{N}}

\def\R{\mathbb{R}}

\def\Z{\mathbb{Z}}

\def\AA{\mathcal{A}}

\def\CC{\mathcal{C}}
\def\DD{\mathcal{D}}

\def\HH{\mathcal{H}}

\def\PP{\mathcal{P}}

\def\aa{\mathbf{a}}
\def\bb{\mathbf{b}}
\def\cc{\mathbf{c}}
\def\ee{\mathbf{e}}

\def\pp{\mathbf{p}}
\def\uu{\mathbf{u}}
\def\vv{\mathbf{v}}

\def\ggg{\mathbf{g}}

\def\xx{\mathbf{x}}

\def\eu{\vartheta}

\def\m#1{\begin{pmatrix}#1\end{pmatrix}}

\def\endo{\text{End}}

\DeclareMathOperator\pd{\partial}
\def\deg{\text{deg}}

\DeclareMathOperator\ord{\text{ord}}

\def\mod{\ \text{mod}\ }

\DeclareMathOperator\Ass{\text{Ass}}
\DeclareMathOperator\std{\text{std}}
\DeclareMathOperator\ass{\text{Ass}}
\DeclareMathOperator\supp{\text{Supp}}
\DeclareMathOperator\frob{\textup{Frob}}
\newcommand\gaps{\textup{gaps}}
\newcommand\diff{\textup{diff}}
\newcommand\sepdeg{\textup{sepdeg}}

\newtheorem{thm}{Theorem}[section]
\newtheorem*{thm*}{Theorem}
\newtheorem{lemma}[thm]{Lemma}
\newtheorem{cor}[thm]{Corollary}
\newtheorem{prop}[thm]{Proposition}

\newtheorem{setup}[thm]{Setup}

\theoremstyle{definition}
\newtheorem{defn}[thm]{Definition}
\newtheorem{ex}[thm]{Example}

\theoremstyle{remark}
\newtheorem{rem}[thm]{Remark}

\numberwithin{equation}{section}

\newcommand{\blank}[1]{}

 \title{Differential powers of monomial ideals in affine semigroup rings}

     \author{Soumyadeep Misra}
     \author{Christopher Wong}
     \address{1460 Jayhawk Blvd, Lawrence, KS, 66044, USA}
     \email{cywong@ku.edu}
     \email{soumyadeep@ku.edu}
     \thanks{The second author was partially supported by the standard NSF grant DMS-2302685 and the NSF CAREER grant DMS-1945611}

     \date{October 21st 2026}

     \keywords{D-module, Differential powers, Affine semigroup, Symbolic powers}
     \subjclass{13N10 (primary), 13F65 (secondary)}

\begin{document}

\begin{abstract}
This article gives explicit descriptions for differential powers of monomial ideals in certain (possibly non-normal) affine semigroup rings. We use these descriptions to show that the topologies defined by the ordinary, symbolic, and differential powers of any monomial ideal are all linearly equivalent. We then prove the existence of differential signature in our settings, relating it to the volume of some convex polytope. Finally, we apply our results to pinched Veronese rings.
\end{abstract}

\maketitle 

\section{Introduction}

The differential powers of an ideal $I\subseteq R$, denoted by $I^{\langle n\rangle}$ for some $n\ge 0$, were first studied for their connection to symbolic powers, as well as their ability to characterize regularity of local rings. They also play a key role in defining \emph{differential signature}, a characteristic free analogue of $F$-signature, which has been used to quantify singularities of local rings, \cite{DDSG+18}, \cite{BJNB19}. While explicit computations of differential powers are scarce in the literature, there has been recent study into the number of generators of $I^{\langle n\rangle}$ for certain monomial ideals in a polynomial ring as $n$ grows, \cite{KMP+26}. 

This article considers another asymptotic measure of differential powers: when is the differential power topology determined by $\{I^{\langle n\rangle}\}_{n\ge 0}$ and the $I$-adic topology determined by $\{I^n\}_{n\ge 0}$ equivalent? A positive answer could be interpreted to mean that differential and ordinary powers of an ideal behave similarly for $n\gg 0$. 

This question follows an analagous one for \emph{symbolic powers}, $I^{(n)}$. Indeed, for any ideals $I$ in a finitely generated $\k$-algebra, we have the pair of containments,
\[
I^n\subseteq I^{(n)}\subseteq I^{\langle n\rangle},
\]
so showing that the differential power topology and $I$-adic topology are equivalent is sufficient to show that the symbolic and $I$-adic topologies are equivalent. Conversely, settings when the symbolic topology and $I$-adic topology are equivalent suggests an analogous result for differential powers. 

We know that the symbolic and $I$-adic topologies are equivalent if and only if $I^{(Cn)}\subseteq I^n$ for some $C>0$ and all $n\ge 0$; moreover, such a $C$ exists for every prime ideal $P$ of a complete local domain, \cite{HKV09}, \cite{Swa00}. Computations of this constant $C$ are scarce in the literature, and we provide concrete examples of this constant in this article.

This article will focus on monomial ideals in two classes of (possibly non-normal) semigroup rings: $\CC$-semigroup rings and scored semigroup rings. These are both obtained by taking a saturated affine semigroup and removing relatively few points, and can be considered as a higher dimensional analogue of \emph{numerical semigroups}, \Cref{scored defn}. We find an explicit description of differential operators of these semigroup rings which we use to derive an explicit description of the differential powers of their monomial ideals. More precisely, $I^{\langle n\rangle}$ is a monomial ideal, and membership of some monomial in $I^{\langle n\rangle}$ is determined by whether some explicitly computable function, $\lambda_I(\uu)$, is greater than or equal to $n$, \Cref{membership criterion}. If $I$ is prime and \emph{not} the maximal ideal of a $\CC$-semigroup, this function $\lambda_I(\uu)$ is linear, so our membership criterion reduces to checking if $\uu$ lies in a single half-space, \Cref{scored diff powers}, \Cref{C diff powers}, \Cref{scored diff powers prime}, \Cref{C semigroup diff powers prime}. 

We then use our description of differential powers to study the equivalence of the $I$-adic, symbolic, and differential power topologies. The following theorem is stated as \Cref{:t prime containment} below.

\begin{thm*}
Suppose that either $R$ is a scored semigroup ring and $P\subseteq R$ a prime monomial ideal, or $R$ is a $\CC$-semigroup ring and $P\subseteq R$ a \emph{non-maximal} prime monomial ideal. Then there exists some $D>0$ such that for all $n\ge 0$,
\[
P^{(Dn)}\subseteq P^{\langle Dn\rangle}\subseteq P^n.
\]
If $\mm$ is the homogeneous maximal ideal of a $\CC$-semigroup ring $R$, then for $n$ sufficiently large $\mm^{((D+1)n)}\subseteq \mm^{\langle (D+1)n\rangle}\subseteq \mm^n$.
\end{thm*}

In the above theorem, $D$ is an explicit integer we can explicitly calculate in terms of the generators of our semigroup ring and our \emph{support function}, \Cref{support function}. In particular $D$ is independent of the ideals $I$ in our following statements. We also show this constant $D$ cannot be taken to be smaller, giving an analogous result for symbolic powers, \cite{Wal18}. 

We use this to show the following, written as \Cref{cofinality theorem} below.

\begin{thm*}
If $R$ is a scored semigroup ring or $\CC$-semigroup ring and $I\subseteq R$ a monomial ideal, there exists a constant $C>0$ such that 
\[
I^{(Cn)}\subseteq I^{\langle Cn\rangle}\subseteq I^n.
\]
\end{thm*}

Turning to the existence of differential signature, we use our description of differential powers to show that $\{I^{\langle n\rangle}\}$ is a \emph{graded family} for all radical ideals $I$ in a scored semigroup ring, \Cref{scored graded diff powers}, correcting an erroneous claim made in \cite[Example 4.35]{BJNB19}. Finally we show that, regardless of gradedness, the differential signature of both $\CC$-semigroup rings and scored semigroup rings exist and can be expressed as the volume of a certain polytope, \Cref{diff sig}. 

We end by applying our results to \emph{pinched Veronese rings}, as in many cases these rings are isomorphic to either $\CC$-semigroups or scored semigroups, \Cref{pinched veronese containment}.

\section{Affine Semigroups}

Given a $d\times r$ matrix $A$ with integer entries, we let $\N A\subseteq \N^d$ denote the (affine) semigroup generated by the columns of $A$ which we denote by $\aa_1,\cdots,\aa_r$. Given an affine semigroup $\N A$ and a field $\k$, the associated semigroup ring, $\k[\N A]$, is the $\k$-algebra $
\k[\N A]=\k[x^\uu:\uu\in \N A]=\k[x^{\aa_1},\cdots,x^{\aa_r}]$.
Going forward we will assume that $A$ has full rank, (so $d\le r$) and $\aa_1,\cdots,\aa_r$ are $\N$-linearly independent.

Let $\Z A$ denote the $\Z$-span of the columns of $A$. If
$\Sigma$ is the Smith normal form of $A$ then $\Z A\simeq \Z \Sigma=\simeq n_1 \Z \oplus \cdots \oplus n_d \Z$, for some positive integers $n_d$. Multiplying each summand by $\frac{1}{n_i}$ preserves the isomorphism type of $\Z A$, so we can assume without loss of generality that $\Z A=\Z^d$. 

Consider the (rational polyhedral) cone $\R_{\ge 0} A=\{\sum_i \alpha_i \aa_i:\alpha_i \in \R_{\ge 0}\}$. A \emph{face} of $\R_{\ge 0} A$ is of the form $\R_{\ge 0}A\cap \HH$ for some hyperplane $\HH$. In particular a face is closed under nonnegative scalar multiplication and addition. A \emph{facet} is a face of codimension 1, and we let $\HH_1,\cdots,\HH_m$ be the hyperplanes in $\R^d$ such that $\R_{\ge 0}A\cap \HH_1,\cdots,\R_{\ge 0}\cap \HH_m$ are the facets of $\R_{\ge 0}A$. For each facet $\R_{\ge 0}A\cap \HH_i$ there exists a unique linear form $h_i$ such that
\[
h_i(\HH_i)=0, \ \  h_i(\Z A)=\Z, \text{ and }\ h_i(\R_{\ge 0} A)\ge 0.
\]
The coefficients of this $h_i$ can be obtained by taking the coefficients of some vector $\vv$ perpendicular to $\HH_i$ and scaling appropriately, in which case $h_i(\uu)=\langle \vv,\uu\rangle$. We call $h_1,\cdots,h_m$ the \emph{(primitive integral) defining functionals} of both $\R_{\ge 0} A$ and $\N A$, and $\HH_1,\cdots,\HH_m$ the \emph{defining hyperplanes} of both $\R_{\ge 0}A$ and $\N A$. In the literature, these defining functionals are often defined as a generating set of the dual cone of $\R_{\ge 0}A$.

Any face of $\R_{\ge 0}A$ is the intersection of the facets that contain $F$. That is, $F=\{\uu\in\R_{\ge 0} A: h_i(\uu)=0 \text{ for all } \HH_i\supseteq F\}$. The \emph{support} of a face $F$ is the set $\supp(F)=\{i\in\N: \HH_i\supseteq F\}$. 

A face of the semigroup $\N A$ is of the form $F\cap \N A$ where $F$ is a face of $\R_{\ge 0} A$. 
Going forward we will abuse notation and let $F$ denote either a face of $\N A$ or $\R_{\ge 0} A$ depending on the context. 

\begin{defn}
\label[defn]{support function}
For some face $F$ of $\N A$ we let the \emph{support function} of $F$, denoted by $\phi_F$, be the linear functional $
\phi_F=\sum_{i\in \supp(F)} h_i.
$
\end{defn}

If $\uu\in \N A$ then $\phi_F(\uu)=0$ if and only if $\uu\in F$. Further, $\phi_{\HH_i}=h_i$ and $\phi_{\{0\}}=\phi_{0}=h_1+\cdots+h_m$. 

\begin{rem}
\label[rem]{big supp function}
Let $\N A$ be a semigroup with defining functionals $h_1,\cdots,h_m$, and $F$ a proper face of $\N A$. We claim that for all collections of integers $\{N_i:i\notin \supp(F)\}$ there exists $\xx\in F$ such that $h_i(\xx)>N_i$. Indeed, this statement is vacuously true if $F=\{0\}$. Now suppose there exists some $\HH_i\not\supseteq F$. For each $\HH_i\not\supseteq F$ choose $\xx_i\in F\setminus \HH_i$ and let $\xx=\sum_{i\notin\supp(F)} \xx_i$. Then $h_i(\xx)\ge 1$ so $h_i(N\xx)=Nh_i(\xx)\ge N$ for any integer $N$, while $N\xx$ still lies in $F$. 
\end{rem}

The \emph{saturation} of $\N A$, denoted by $\overline{\N A}$, is $\overline{\N A}=\R_{\ge 0} A\cap \Z^d$. If $\N A=\overline{\N A}$ we say $\N A$ is \emph{normal}. Thus if $\N A$ is normal there exist linear forms $h_1,\cdots,h_m$ such that $\displaystyle{
\N A=\bigcap_{i=1}^m\{\uu\in\Z^d: h_i(\uu)\ge 0\}.}$ A well-known fact is that the normalization of $\k[\N A]$ is $\k[\overline{\N A}]$, justifying the use of the term ``normal". 

We now define the two classes of affine semigroups we will study in this article, both obtained by removing relatively few points from their saturation.

\begin{defn}
\label[defn]{scored defn}
A semigroup $\N A$ is \emph{scored} if $\overline{\N A}\setminus \N A$ consists finitely many hyperplane sections of $\overline{\N A}$ parallel to (but not equal to) the defining hyperplanes of $\overline{\N A}$. 

A semigroup $\N A$ is a $\CC$-semigroup if $\overline{\N A}\setminus \N A$ is a finite set which we denote by $\gaps(\N A)$.
\end{defn}

Suppose $\dim(\N A)=1$, so $\overline{\N A}=\N$. Then the defining hyperplane of $\N A$ is the point at the origin, so scored semigroups and $\CC$-semigroups both refer to subsemigroups $\N A$ of $\N$ such that $|\N\setminus \N A|<\infty$. These one dimensional semigroups are called \emph{numerical semigroups}. The \emph{Frobenius number}, denoted by $\frob(\N A)$, is the maximum integer not in $\N A$. 

\subsection{Monomial Ideals of Semigroup Rings}

For the rest of this article, all ideals $I\subseteq \k[\N A]$ will be monomial ideals. We define the \emph{standard points of $I$}, denoted by $\std(I)$, to be $\std(I)=\{\uu\in \N A:x^\uu\notin I\}$. We may write important properties of $I$ in terms of $\std(I)$. For example, $I\subseteq \k[\N A]$ is prime if and only if $\std(I)$ is a face of $\N A$. Thus we will write $P_F$ to be the monomial prime ideal whose set of standard points is $F$. If $I$ is a monomial ideal then so are its associated primes, hence we can write $\ass(I)=\{P_{F_1},\cdots,P_{F_s}\}$ where $F_1,\cdots,F_s$ are faces of $\N A$, \cite[Lemma 7.10 and Proposition 8.11]{MS05}

\begin{thm}
\label[thm]{primary monomial char}
$I$ is $P_F$-primary if and only if there exist $\cc_1,\cdots,\cc_\ell$ such that 
\[
\std(I)=\bigcup_{i=1}^\ell (\cc_i+\Z F)\cap \N A,
\]
\end{thm}

\Cref{primary monomial char} is presented as a corollary to a statement about primary decomposition of \emph{binomial} ideals \cite[Proposition 2.8]{Mil10}. For a proof that restricts itself to monomial ideals, we refer readers to \cite[Theorem 3.8 (2)]{MY22}.

\begin{rem}
\label[rem]{finite support function}
If $I$ is $P_F$-primary then $\phi_F(\std(I))$ is a finite set (of nonnegative integers). Indeed, if we take $\std(I)$ as in \Cref{primary monomial char}, then $\phi_F(\std(I))=\{\phi_F(\cc_1),\cdots,\phi_F(\cc_\ell)\}$.
\end{rem}

We now consider $\std(I)$ when $\N A$ is a numerical semigroup. If $0\ne a\in \N A$, then $\std(x^a\k[\N A])=\N A\setminus (a+\N A)$ (also known as the \emph{Ap\'{e}ry set of $\N A$ relative to $a$}) consists of the least elements of $\N A$ congruent to $i\mod a$ for each $0\le i<a$. In particular, $|\std(x^a\k[\N A])|=a$, \cite[Lemma 2.4]{GSR09}. 

In contrast, if $a\in \Z\setminus \N A$, we can view $x^a\k[\N A]$ as a \emph{fractional ideal} of $\k[\N A]$ and still discuss its set of standard points in $\N A$. If $a\in \N\setminus \N A$ then $\std(x^a\k[\N A])=\N A\setminus (a+\N A)$ contains the least elements of $\N A$ in each congruence class modulo $a$, though there will be other standard points as well. We summarize this discussion below.

\begin{lemma}
\label[lemma]{apery lemma}
If $\N A$ is a numerical semigroup and $a\in \Z$, then 
\[
|\std(x^a\k[\N A])|=|\N A\setminus(a+\N A)|\ge a,
\]
with equality holding if $a\in \N A$. 
\end{lemma}

\begin{ex}
Let $A=\m{ 2 & 3}$ so $\N A=\langle 2,3\rangle=\{0,2,3,4,\cdots\}$ and $\k[\N A]=\k[x^2,x^3]$. We have $\gaps(\N A)=\{1\}$, $\frob(\N A)=1$, and
\[
\std(x^3\k[\N A])=\{0,2,4\}, \quad \std(x\k[\N A])=\{0,2\}, \quad \std(x^{-1}\k[\N A])=\{0\},
\]
and $\std(x^{-a}\k[\N A])=\emptyset$ for all $a\ge 2$, verifying \Cref{apery lemma}.\end{ex}

If $\N A$ is scored then $h_i(\N A)$ is a numerical semigroup for every defining functional $h_i$. Thus much study of scored semigroups boil down to studying numerical semigroups.

\begin{ex}
\label[ex]{running ex 1}
Consider the semigroups defined by the matrices
\[
A=\m{ 2 & 1 & 1\\ 1 & 1 & 2} , \quad B=\m{3 & 2 & 2 & 1\\2 & 2 & 1 & 2} ,\quad C=\m{5 & 4 & 3 & 3 & 3 & 2 & 1 \\ 10 & 8 & 6 & 5 & 4 & 1 & 1}.
\]
Also consider the ideals $I=\langle x^2y^2,x^3y^2,x^4y^2\rangle$ of $\k[\N A]$ and $J=\langle x^2y^2,x^3y^2,x^4y^2\rangle$ of $\k[\N B]$. In \Cref{fig:normal_ex} and \Cref{fig:scored_ex} the elements of $\std(I)\subseteq \N A$, $\std(J)\subseteq \N B$ and $\N C$ are drawn using black dots, while $I$ and $J$ are represented by blue squares. In \Cref{fig:C_ex} the points of $\N C$ are drawn in black dots with the gaps denoted by open dots.

\begin{figure}[ht]

\begin{subfigure}{0.3\textwidth}
\includegraphics[alt={Three pictures of lattice points, each describing an affine semigroups in the example.},width=0.85\linewidth, height=3.75cm]{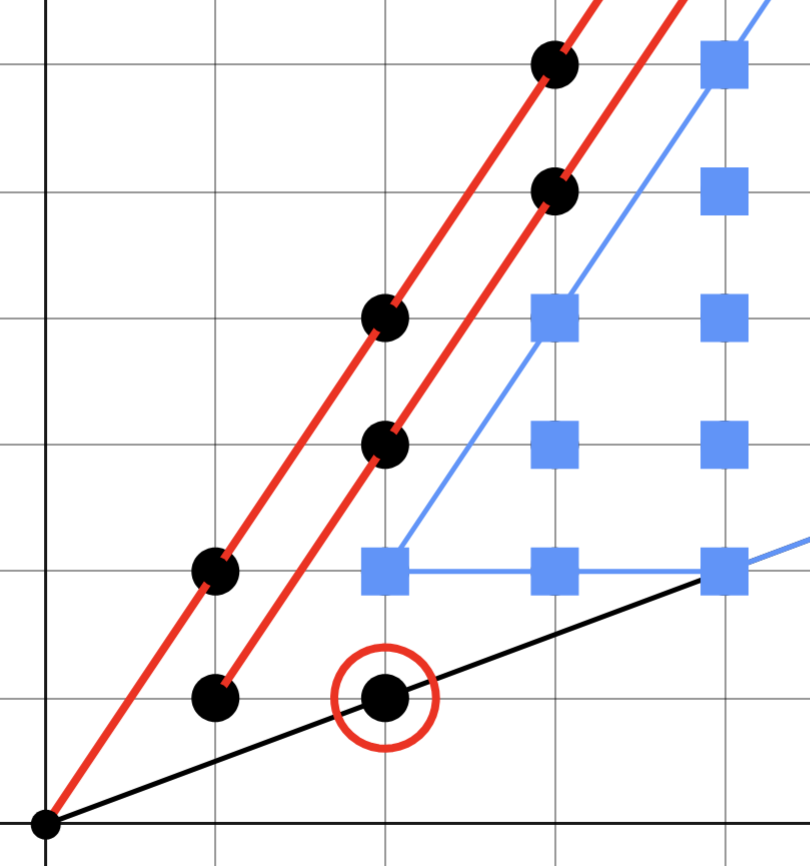} 
\caption{$I\subseteq \k[\N A]$,\\ $\N A$ normal }
\label{fig:normal_ex}
\end{subfigure}
\begin{subfigure}{0.3\textwidth}
\includegraphics[width=0.85\linewidth, height=3.75cm]{normal_ex_ideal} 
\caption{$J\subseteq \k[\N B]$,\\ $\N B$ scored }
\label{fig:scored_ex}
\end{subfigure}
\begin{subfigure}{0.3\textwidth}
\includegraphics[width=0.85\linewidth, height=3.75cm]{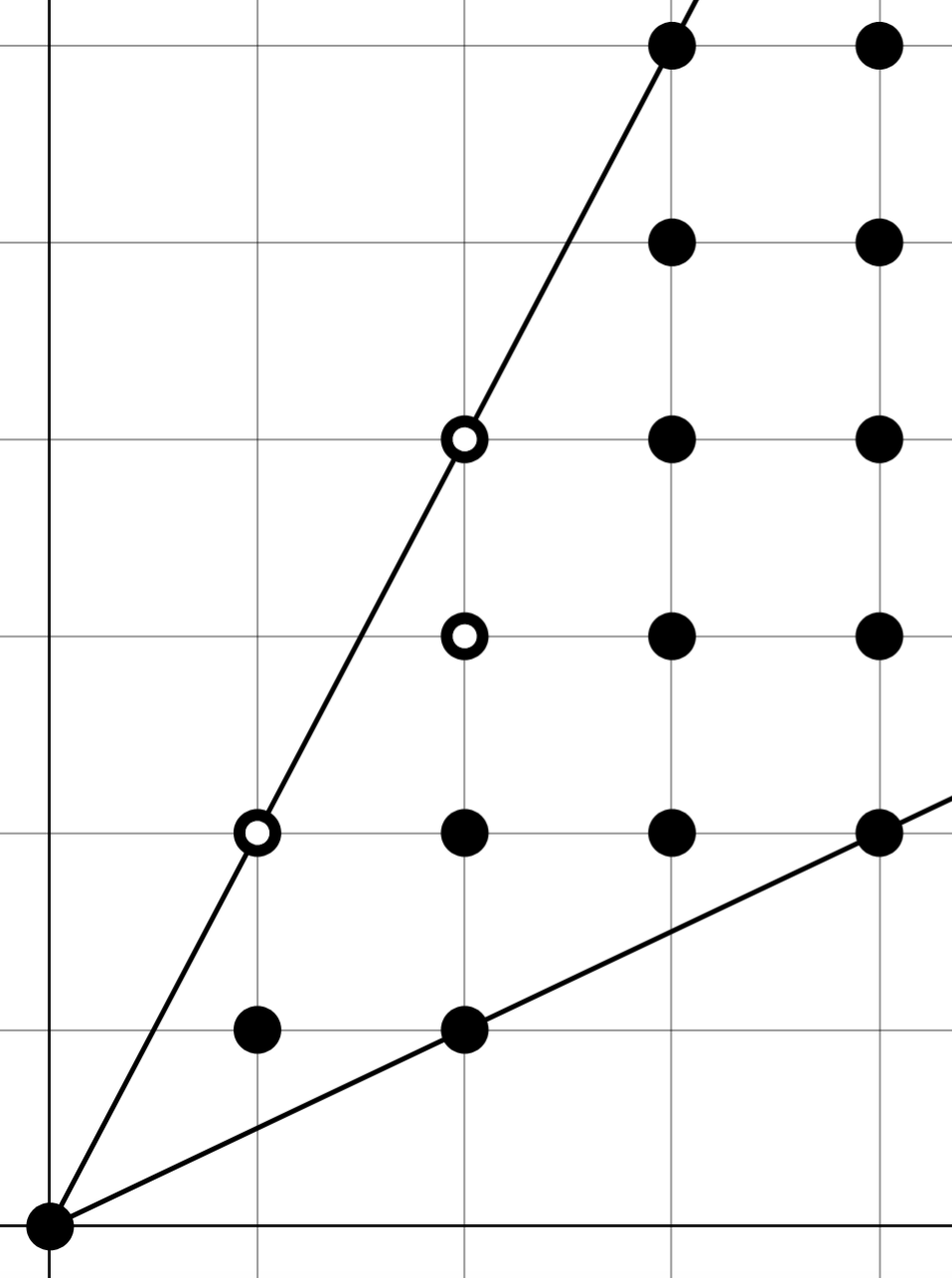} 
\caption{$\k[\N C]$,\\ $\N C$ a $\CC$-semigroup}
\label{fig:C_ex}
\end{subfigure}

\caption{Examples of ideals in two dimensional affine semigroups.}
\label{fig:image1}
\end{figure}

The semigroup $\N A$ is normal with defining functionals $h_1(x,y)=2x-y$ and $h_2(x,y)=2y-x$. We may also write $\N A=\{\uu\in \Z^2:h_1(\uu)\ge 0,\text{ and }h_2(\uu)\ge 0\}$. 

The semigroup $\N B$ is scored, where $\overline{\N B}\setminus \N B=\{\uu:h_1(\uu)=1\}$. Thus $h_1(\N B)=\N\setminus \{1\}=\langle 2,3\rangle$ and $h_2(\N B)=\N$. The semigroup $\N C$ is a $\CC$-semigroup, where $\gaps(\N C)=\{(1,2),(2,3),(2,4))\}$. It is evident that $\N A$ is the saturation of both $\N B$ and $\N C$. 

Turning to monomial ideals, if we let $F_0=\{0\}$, $F_1=\R_{\ge 0}(1,2)^T$ and $F_2=\k (2,1)^T$ then the monomial prime ideals of $\k[\N A]$ are
\[
P_{F_0}=\langle x^2y,xy,xy^2\rangle=\mm, \quad P_{F_1}=\langle xy,x^2y\rangle, \quad P_{F_2}=\langle xy^2,xy\rangle.
\]
We can write $\std(I)=((0,0)+F_1)\cup ((1,1)+F_1)\cup((2,1)+F_0)$
which shows that $I$ is not primary and $\ass(I)=\{P_{F_1},P_{F_0}\}$. Also, $\phi_{F_1}(\std(I))$ is finite which shows the converse to \Cref{finite support function} is false. On the other hand, we can write $\std(J)=((0,0)+F_1)\cup((2,1)+F_1)$, also illustrated with red lines in \Cref{fig:scored_ex}, confirming that $J$ is $P_{F_1}$-primary by \Cref{primary monomial char}.
\end{ex}

\section{Differential Operators}

\subsection{Differential Powers}

Given a field $\k$ and a $\k$-algebra $R$, we can define the ring of \emph{$\k$-linear differential operators over $R$}, $\DD_{R|\k}\subseteq \endo_\k(R)$, inductively as follows: 

\begin{itemize}

\item Elements in $\DD_{R|\k}^0$ are multiplication by elements in $R$. We abuse notation and let $r\in \DD_{R|\k}^0$ denote the ``multiplication by $r$" map. 

\item For $n> 0$  
\[
\DD_{R|\k}^{n}:=\{\delta\in \endo_\k(R):[\delta,r]\in \DD_{R|\k}^{n-1}\text{ for all } r\in \DD_{R|\k}^0\},
\]
where $[\delta,r]=\delta r-r\delta$ is the commutator of $\delta$ and $r$.

\end{itemize}

We say $\DD_{R|\k}^n$ is the set of operators of \emph{order at most} $n$, and define
\[
\DD_{R|\k}=\DD_R:=\bigcup_{n\ge 0} \DD_{R|\k}^n.
\]

One can show that the collection $\{\DD_{R|\k}^n\}_{n\ge 0}$ forms an ascending filtration for $\DD_{R|\k}$.

\begin{defn}
Let $\k$ be a field, $R$ a finitely generated $\k$-algebra and $I\subseteq R$ an ideal. We declare $I^{\langle 0\rangle}=R$ and for $n\ge 1$ the \emph{$n^{th}$ differential power of $I$} is
\[
I^{\langle n\rangle_\k}=I^{\langle n\rangle}=\{f:\DD_{R|\k}^{n-1}\cdot f\subseteq I\}.
\]
\end{defn}

\begin{lemma}
\label[lemma]{differential lemma}
Let $R$ be a finitely generated $\k$-algebra and $I$ an ideal of $R$. Then for any $n\ge 0$,
\begin{enumerate}

\item $I^n\subseteq I^{\langle n\rangle}$

\item $\left(\bigcap_{\alpha\in \AA} I_\alpha\right)^{\langle n\rangle}=\bigcap_{\alpha\in \AA} (I_\alpha)^{\langle n\rangle}$.

\item If $n\ge 1$ and $I$ is $P$-primary then so is $I^{\langle n\rangle}$. 
\end{enumerate}
\end{lemma}

\begin{proof}
The first two points are Theorem 2.5 and Exercise 2.13 of \cite{DDSG+18} respectively. Though \cite[Proposition 2.6]{DDSG+18} is only proven for prime ideals $P$, assuming $P$ is primary and following their proof yields the third point.
\end{proof}

The initial interest in differential powers arose out of their connection to \emph{symbolic powers} which we define below.

\begin{defn}
Let $\k$ be a field, $R$ a finitely generated $\k$-algebra and $I$ an ideal of $R$. Then for any $n\ge 0$, the \emph{$n^{th}$ symbolic power} of $I$ is
\[
I^{(n)}=\bigcap_{P\in \Ass(I)}I^nR_P\cap R.
\]
\end{defn}

\begin{rem}
\label[rem]{symbolic lemma}
It follows by definition that $I^n\subseteq I^{(n)}$ and $\{I^{(n)}:n\ge 0\}$ is a graded family, i.e., $I^{(n)}I^{(m)}\subseteq I^{(n+m)}$ for integers $n,m\ge 0$. One can also show that if $I$ is $P$-primary then $I^{(n)}$ is the smallest $P$-primary ideal containing $I^n$. Combining this with \Cref{differential lemma} (2) and (3) shows that for any ideal $I$,
\[
I^n\subseteq I^{(n)}\subseteq I^{\langle n\rangle}.
\]
Thus showing that for any integer $n$ there exists an integer $n'$ such that $I^{\langle n'\rangle}\subseteq I^n$ is enough to show that the $I$-adic, symbolic and differential power topologies are cofinal.
\end{rem}

\subsection{Differential Operators on Semigroup Rings}

We begin by paraphrasing \cite[Theorem 16.11.2]{Gro67} to discuss $\DD_{\k[\Z^d]|\k}$. Moving forward we will assume $\k$ is a field of characteristic zero.

\begin{ex}
\label[ex]{poly ring diff ops}
Let $R=\k[\Z^d]=\k[x_1^{\pm 1},\cdots,x_d^{\pm 1}]$. We let $\pd_i\in \endo_\k(R)$ denote the formal partial differentiation operators $\frac{\pd}{\pd x_i}$ for $1\le i\le d$, and let $\eu_i:=x_i\pd_i$ denote the $i^{\text{th}}$ Euler operator. Then for each $n\ge 0$,
\[
\DD_{R|\k}^n=\bigoplus_{|\uu|\le n} R\pd^{\uu}=\bigoplus_{|\uu|\le n} R\eu^{\uu}.
\]
As a result, $\DD_{R|\k}=R\langle \eu \rangle$, where $R\langle \eu\rangle$ denotes the smallest subalgebra of $\DD_{R|\k}$ containing $R$ and $\eu_1,\cdots,\eu_d$.
 
We now consider $R$ equipped with the standard $\Z^d$ grading where $\deg(x_i)=\ee_i$. This induces a $\Z^d$ grading on $\DD_{R|\k}$ where $\deg(\eu_i)=\mathbf{0}$ (equivalently, $\deg(\pd_i)=-\ee_i$). Given $\bb\in \Z^d$, because $\DD_{R|\k}=R\langle \eu\rangle$ we have,
\begin{equation}
\label{graded laurent diff ops}
[\DD_{R|\k}]_\bb=[R]_\bb\langle \eu\rangle=x^\bb\k[\eu_1,\cdots,\eu_d].
\end{equation}
Since the $\pd_1,\cdots,\pd_d$ are pairwise commutative so are the Euler operators $\eu_1,\cdots,\eu_d$. Furthermore, $\k$-linearity allows us to treat $\k[\eu_1,\cdots,\eu_d]=\k[\eu]$ as a commutative polynomial ring in $d$ indeterminates.
\end{ex}

We now consider differential operators over general affine semigroups $\k[\N A]$, following \cite{ST01}. Though the following discussion generalizes to formally \'{e}tale maps we restrict to the case of localizations, \cite[Theorem 2.2.5, Corollary 2.2.6]{Mas91}. 

The localization map $\k[\N A]\hookrightarrow \k[\Z^d]$ induces an (order preserving) injective map $\DD_{\k[\N A]}\hookrightarrow \DD_{\k[\Z^d]}$. By identifying $\DD_{\k[\N A]}$ with its image we may consider $\DD_{\k[\N A]}$ as a subalgebra of $\DD_{\k[\Z^d]}$. We can also show that
\[
\DD_{\k[\N A]}=\{\delta\in\DD_{\k[\Z^d]}:\delta(\k[\N A])\subseteq \k[\N A]\}.
\]
If $\DD_{\k[\Z^d]}$ is graded as in \Cref{poly ring diff ops}, then $\DD_{\k[\N A]}$ inherits this grading. Explicitly, for any $\bb\in \Z^d$, $[\DD_{\k[\N A]}]_\bb$ consists of operators of the form $x^\bb g(\eu)\in x^\bb\k[\eu]$ such that for all $\uu\in\N A$, 
\begin{equation}
\label{semigroup diff ops eq}
x^\bb g(\eu)\cdot x^\uu\in \k[\N A].
\end{equation}
The following lemma describes the action of $\k[\eu]$ on $\k[\Z^d]$.

\begin{lemma}
\label[lemma]{eu lemma}
Let $g(\eu)\in \k[\eu]$ and $x^\uu\in \k[\Z^d]$. Then $g(\eu)\cdot x^\uu=g(\uu)x^\uu$.
\end{lemma}

\begin{proof}
Because the $\eu_1,\cdots,\eu_d$ are pairwise commutative it suffices to consider the one dimensional case. By $\k$-linearity we may assume that $g(\eu_1)=\eu_1^j$ for some $j\ge 0$. But then our statement follows by induction on $j$.
\end{proof}

Combining \Cref{eu lemma} with \eqref{semigroup diff ops eq}, shows that $x^\bb g(\eu)$ lies in $[\DD_{\k[\N A]}]_\bb$ if and only if for all $\uu\in \N A$, $g(\uu)x^{\bb+\uu}\in \N A$. This occurs if and only if $g(\eu)$ vanishes on $\std(x^{-\bb}\k[\N A])=(\N A\setminus -\bb+\N A)$.

We summarize our discussion below.

\begin{prop}
\label[prop]{semigroup diff ops}
If $\N A$ is a semigroup then $\DD_{\k[\N A]}$ is $\Z^d$-graded, where for each $\bb\in \Z^d$,
\[
[\DD_{\k[\N A]}]_\bb=x^\bb\I(\N A\setminus(-\bb+\N A))\k[\eu],
\]
where $\I(\N A\setminus(-\bb+\N A))$ is the vanishing ideal of $\N A\setminus(-\bb+\N A)$ in $\k[\eu]$.
\end{prop}

Note that $\Z^d\setminus (-\bb+\Z^d)=\emptyset$ for all $\bb\in \Z^d$, hence \Cref{semigroup diff ops} recovers \eqref{graded laurent diff ops}.

We now use \Cref{semigroup diff ops} to give an explicit description of the ring of differential operators on scored semigroups.

\begin{thm}
\label[thm]{:t scored diff ops}
Let $\N A$ be a scored semigroup ring with defining functionals $h_1,\cdots,h_m$ and defining hyperplanes $\HH_1,\cdots,\HH_m$. Then $\I(\N A\setminus(-\bb+\N A))=\eta_{\bb}(\eu)\k[\eu]$ where 
\[
\eta_\bb(\eu):=\prod_{h_i(\bb)\notin h_i(\N A)}\prod_{\rho}(h_i(\eu)-\rho),
\]
where $\rho$ ranges over $h_i(\N A)\setminus(-h_i(\bb)+h_i(\N A))$. Thus $[\DD_{\k[\N A]}]_\bb=x^\bb\eta_\bb(\eu)\k[\eu]$ and $\ord(x^\bb\eta_\bb(\eu))$ is $\sum_{h_i(\bb)\notin h_i(\N A)}|h_i(\N A)\setminus -h_i(\bb)+h_i(\N A)|$.
\end{thm}

\begin{proof}
Given $\bb\in \Z^d$ and $\uu\in \N A$, $\uu\in \N A\setminus(-\bb+\N A)$ if and only if $h_i(\uu+\bb)\notin h_i(\N A)$ for some $i$, hence $\N A\setminus(-\bb+\N A)=\bigcup_{i=1}^m \{\uu\in \N A: h_i(\uu)\notin -h_i(\bb)+h_i(\N A)\}$. Since the Zariski closure of this set is $\bigcup_{\rho}\{\uu\in\k^d:h_i(\uu)=\rho\}$ where $\rho$ ranges over $h_i(\N A)\setminus (-h_i(\bb)+h_i(\N A))$, we have
\[
\I(\{\uu\in \N A: h_i(\uu)\notin -h_i(\bb)+h_i(\N A)\})=\bigcap_{\rho}(h_i(\eu)-\rho)\k[\eu]=\prod_{\rho}(h_i(\eu)-\rho)\k[\eu],
\]
where $\rho$ ranges over $h_i(\N A)\setminus(-h_i(\bb)+h_i(\N A))$.
We replace the intersection with a product in the last equality because we are taking the intersection of principal ideals (in a polynomial ring) whose generators are relatively prime.

Taking the ideal of this union shows that $\I(\N A\setminus (-\bb+\N A))$ is
\begin{equation}
\label{:e scored diff ops}
\bigcap_{i=1}^m \I(\{\uu:h_i(\uu)+h_i(\bb)\notin h_i(\N A)\})=\prod_{h_i(\bb)\notin h_i(\N A)} \prod_{\rho }(h_i(\eu)-\rho)\k[\eu],
\end{equation}
where $\rho$ ranges over $h_i(\N A)\setminus(-h_i(\bb)+h_i(\N A))$. It suffices to take the second product over those $1\le i\le m$ for which $h_i(\bb)\notin h_i(\N A)$ because if $h_i(\bb)\in h_i(\N A)$ then $h_i(\N A)\setminus-h_i(\bb)+ h_i(\N A)=\emptyset$.

The second statement follows immediately from the fact that $\ord(\eu_{x_i})=1$, hence $\ord(x^\bb\eta_\bb(\eu))=\deg(\eta_\bb(\eu))$.
\end{proof}

We may also specialize the above theorem to normal semigroups, recovering \cite[Theorem 3.2.2]{ST01}.

\begin{prop}
\label[prop]{C semigroup diff ops}
If $\k[\N A]$ is a $\CC$-semigroup then $\I(\N A\setminus(-\bb+\N A))=\overline{\eta_{\bb}}(\eu)\I(\Theta(\bb))$, 
where
\[
\overline{\eta_\bb}(\eu)=\prod_{h_i(\bb)<0} \prod_{j=0}^{-h_i(\bb)-1} (h_i(\eu)-j),
\]
and $\Theta_{\N A}(\bb)=\Theta(\bb)=\{\uu\in \N A:\bb+\uu\in \gaps(\N A)\}$.
\end{prop}

\begin{proof}
If $\N A$ is a $\CC$-semigroup then $\uu\notin \N A$ if and only if $h_i(\uu)<0$ for some $1\le i\le m$ or $\uu\in \gaps(\N A)$. Thus for some $\bb\in \Z^d$,
\[
\N A\setminus(-\bb+\N A)=\left(\bigcup_{i=1}^m \{\uu: \uu:0\le h_i(\uu)<-h_i(\bb)\}\right)\cup \{\uu:\uu+\bb\in \gaps(\N A)\}.
\]
By analagous logic to the proof of \Cref{:t scored diff ops}, the ideal of the first set in the union is generated by $\overline{\eta_\bb}(\eu)$. Taking the ideal of the (disjoint) union gives us the desired form of $\I(\N A\setminus(-\bb+\N A))$.
\end{proof}

If $\k[\N A]$ is a $\CC$-semigroup then by \Cref{:t scored diff ops}, $\I(\overline{\N A}\setminus(-\bb+\overline{\N A}))=\overline{\eta_\bb}(\eu)\k[\eu]$.

\section{Computing Differential Powers}

We first note that if $\DD_R$ is homogeneous (as in the case of affine semigroup rings), then to show $f\in I^{\langle n\rangle}$ it suffices to show that $\delta(f)\in I$ for all \emph{homogeneous} $\delta\in \DD_R^{n-1}$. Moreover,

\begin{lemma}
\label[lemma]{diff lemma 2}
Let $\k[\N A]$ be an affine semigroup ring over a field of characteristic zero and $I\subseteq \k[\N A]$ a monomial ideal. Then $I^{\langle n\rangle}$ is also a monomial ideal.
\end{lemma}

\begin{proof}
Consider $f=\sum_\uu c_\uu x^\uu\in \k[\N A]$ with $c_\uu\in \k^*$ and assume $f\in I^{\langle n\rangle}$. Let $\delta \in \DD_{\k[\N A]}^{n-1}$ and suppose $\deg(\delta)=\bb$. By \Cref{semigroup diff ops} we can write $\delta=x^\bb g(\eu)$ for some $g(\eu)\in \I(\N A\setminus(-\bb+\N A))$. Since $f\in I^{\langle n\rangle}$ we have
\[
\delta(f)=\sum_\uu c_\uu \delta(x^\uu)=\sum_\uu c_\uu g(\uu)x^{\uu+\bb}\in I,
\]
where we may assume $g(\uu)\ne 0$ for each $\uu$ in the above sum. Because $I$ is a monomial ideal, $g(\uu)x^{\uu+\bb}=\delta\cdot (x^\uu)\in I$. Thus every supporting monomial of $f$ lies in $I^{\langle n\rangle}$ and we are done.
\end{proof}

\begin{rem}
\label[rem]{membership criterion}
Consider a homogeneous operator of degree $\bb$ of the form $\delta=x^\bb g(\eu)$ for some $g(\eu)\in \I(\N A\setminus(-\bb+\N A))$ and some monomial $x^\uu$. Then 
\[
\delta(x^\uu)=g(\uu)x^{\bb+\uu}\notin I
\]
if and only if $\bb+\uu=\cc$ for some $\cc\in \std(I)$ and $g(\uu)\ne 0$. Replacing $\bb$ with $\cc-\uu$ in the above equation shows that $x^\uu\notin I^{\langle n+1\rangle}$ if and only if there exists some $\cc\in\std(I)$ and $g\in \I(\N A\setminus(-(\cc-\uu)+\N A))$ of degree at most $n$ such that $g(\uu)\ne 0$. Furthermore, if we define $\lambda_I(\uu,\N A)=\lambda_I(\uu)$ to be
\[
\lambda_I(\uu)=\min_{\cc\in\std(I)}\min\{\deg(g(\eu)):g(\eu)\in \I(\N A\setminus(-(\cc-\uu)+\N A)),g(\uu)\ne 0\},
\]
then $x^\uu\in I^{\langle\lambda_I(\uu)\rangle}\setminus I^{\langle \lambda_I(\uu)+1\rangle}$. Combining this with \Cref{diff lemma 2}, we have $I^{\langle n\rangle}=\langle x^\uu:\lambda_{I}(\uu,\N A)\ge n\rangle$.
\end{rem}

We now compute $\lambda_{I}(\uu,\N A)$ for scored and $\CC$-semigroups. For the rest of the article we make the following assumptions.

\begin{setup}
\label[setup]{semigroup setup}
$\k$ will denote a field of characteristic zero and all ideals will be nonzero monomial ideals. $\N A$ will denote an affine semigroup of full rank with $\Z A=\Z^d$, defining functionals $h_1,\cdots,h_m$ and defining hyperplanes $\HH_1,\cdots,\HH_m$. We will assume $\CC$-semigroups are of dimension at least 2 and have at least 1 gap, otherwise they fall under the scored case.
\end{setup}

\subsection{Scored Semigroups}

Normal semigroups are trivially scored, where each $h_i(\N A)=\N$, and numerical semigroups are scored semigroups of dimension one. Thus any statement about scored semigroups also applies to normal and numerical semigroups. 

\begin{defn}
Let $\N A$ be a scored semigroup. For each $1\le i\le m$ and $\bb\in \Z^d$ we let 
\[
\sigma_i(\bb):=h_i(\N A)\setminus (-h_i(\bb)+h_i(\N A)).
\]
\end{defn}

We may think of $\sigma_i(\bb)$ as the standard points (in $h_i(\N A)$) of the fractional ideal $t^{-h_i(\bb)}\k[h_i(\N A)]\subseteq \k[t^{\pm 1}]$. \Cref{apery lemma} states that $|\sigma_i(\bb)|\ge -h_i(\bb)$ with equality holding if $-h_i(\bb)\in h_i(\N A)$. We may write $\eta_\bb$ in \Cref{:t scored diff ops} as ${\eta_\bb(\eu)=\prod_{1\le i\le m} \prod_{\rho\in \sigma_i(\bb)}(h_i(\eu)-\rho})$ and $\deg(\eta_\bb)=\sum_{i=1}^m |\sigma_i(\bb)|$. We also observe that if $h_i(\xx)=0$ then $\sigma_i(\bb)=\sigma_i(\xx+\bb)$ for all $\bb\in \Z^d$. This occurs if $\xx$ lies in some face $F$ and $i\in \supp(F)$.

\begin{thm}
\label[thm]{scored primary diff powers}
Let $I$ be a $P_F$ primary ideal of a scored semigroup ring $\k[\N A]$. Then 
\begin{equation}
\label{scored primary eq}
\lambda_{I}(\uu,\N A)=\lambda_I(\uu)=\min_{\cc\in \std(I)} \sum_{\substack{i\in \supp(F)\\h_i(\cc-\uu)\notin h_i(\N A)}}|\sigma_i(\cc-\uu)|.
\end{equation}
\end{thm}

\begin{proof}
Recall that $\eta_{\cc-\uu}(\uu)=0$ if and only if $h_i(\uu)\in h_i(\N A)$ and $h_i(\uu+\cc-\uu)=h_i(\cc)\notin h_i(\N A)$ for some $i$. But $\cc\in \N A$, so we have $h_i(\cc)\in h_i(\N A)$ for all $i$. Thus $\eta_{\cc-\uu}(\uu)\ne 0$. Since $\eta_{\cc-\uu}(\eu)$ has minimal degree in $\I(\N A\setminus (-(\cc-\uu)+\N A))$, $\lambda_I(\uu)=\min_{\cc\in\std(I)} \{\deg(\eta_{\cc-\uu}(\eu)\}$.   

For any $\cc\in\std(I)$, by \Cref{big supp function} we may choose $\xx\in F$ such that $h_i(\cc+\xx-\uu)=h_i(\cc-\uu)$ for all $i\in \supp(F)$ and $h_i(\cc+\xx-\uu)>\frob(h_i(\N A))$ for all $i\notin \supp(F)$. We are also guaranteed that $\cc+\xx\in \std(I)$ by \Cref{primary monomial char}. Then 
\[
\deg(\eta_{\cc+\xx-\uu}(\eu))=\sum_{h_i(\cc+\xx-\uu)\notin h_i(\N A)}|\sigma_i(\cc+\xx-\uu)|=\sum_{\substack{i\in \supp(F)\\h_i(\cc-\uu)\notin h_i(\N A)}}|\sigma_i(\cc-\uu)|,
\]
where the above quantity is at most $\deg(\eta_{\cc-\uu}(\eu))$.
Thus when minimizing $\deg(\eta_{\cc-\uu}(\eu))$ as $\cc$ ranges over $\std(I)$ we may assume that $h_i(\cc-\uu)\in h_i(\N A)$ for all $i\notin\supp(F)$. We obtain the desired equality by counting the factors of $\eta_{\cc-\uu}(\eu)$.
\end{proof}

\begin{rem}
\label[rem]{finite input rem}
Even though $\std(I)$ may be an infinite set, when computing $\lambda_I(\uu)$ in \eqref{scored primary eq} it suffices to consider $\sum_i |\sigma_i(\cc-\uu)|$ for \emph{finitely} many $\cc$. To see this, write $\std(I)=\bigcup_{j=1}^\ell (\cc_j+\Z F)\cap \N A$ for finitely many $\cc_j$ as in \Cref{primary monomial char}. Then $h_i((\cc_j+\Z F)\cap \N A)=h_i(\cc_j)$ for all $i\in \supp(F)$, but these are the only indices we consider when summing $|\sigma_i(\cc-\uu)|$ in \eqref{scored primary eq}. Thus when minimizing $\sum_i |\sigma_i(\cc-\uu)|$ it suffices to take the minimum over $\cc\in \{\cc_1,\cdots,\cc_\ell\}$.
\end{rem}

\begin{cor}
\label[cor]{scored diff powers prime}
If $\N A$ is scored and $P_F$ prime then $\lambda_{P_F}(\uu)=\phi_F(\uu)$. Hence $P_F^{\langle n\rangle}=\langle x^\uu:\phi_F(\uu)\ge n\rangle$.
\end{cor}

\begin{proof}
Since $\std(P_F)=F$, for all $\cc\in F$ and $i\in \supp(F)$, $h_i(\cc-\uu)=-h_i(\uu)$ which lies in $-h_i(\N A)$. Thus,
\[
\lambda_{P_F}(\uu)=\min_{\cc\in F} \sum_{i\in \supp(F)}|\sigma_i(-\uu)|=\sum_{i\in \supp(F)}h_i(\uu),
\]
where the last equality follows by \Cref{apery lemma}. But the right hand side is $\phi_F(\uu)$ as desired.
\end{proof}

\begin{cor}
\label[cor]{scored diff powers bounds}
Let $I$ be a $P_F$-primary ideal of a scored semigroup ring $\k[\N A]$.  If $M_F(I):=\max_{\cc\in\std(I)}\{\phi_F(\cc)\}$ then $\phi_F(\uu)-M_F(I)\le \lambda_I(\uu)\le \phi_F(\uu)$.
\end{cor}

\begin{proof}
Similar logic to the previous proof shows that given $\cc\in F\subseteq \std(I)$, $\lambda_I(\uu)\le \sum_{i\in \supp(F)}|\sigma_i(-\uu)|=\phi_F(\uu)$. 

For the lower bound, we first note that \Cref{finite support function} ensures that $\phi_F(\std(I))$ is bounded so discussing the maximum of $\phi_F(\std(I))$ makes sense. For any $\uu\in \N A$ and $\cc\in \std(I)$, $h_i(\uu-\cc)\le |\sigma_i(\cc-\uu)|$ by \Cref{apery lemma}. Thus
\[
\phi_F(\uu)-M_F(I)\le \phi_F(\uu-\cc)=\sum_{\substack{i\in \supp(F)\\h_i(\cc-\uu)\notin h_i(\N A)}}\overbrace{h_i(\uu-\cc)}^{\le |\sigma_i(\cc-\uu)|}+\overbrace{\sum_{\substack{i\in \supp(F)\\h_i(\cc-\uu)\in h_i(\N A)}}h_i(\uu-\cc)}^{\le 0},
\]
where the above quantity is at most $\lambda_I(\uu)$ by definition.
\end{proof}

\begin{cor}
\label[cor]{scored diff powers}
Let $I$ be a monomial ideal of a scored semigroup ring with primary decomposition $I=Q_1\cap \cdots \cap Q_s$, where for each $1\le k\le s$, $\sqrt{Q_k}=P_{F_{k}}$ for some face $F_k$. Then 
\[
I^{\langle n\rangle}=\langle x^\uu: \lambda_{Q_k}(\uu)\ge n\text{ for all } 1\le k\le s\rangle.
\]
\end{cor}

\begin{proof}
We may assume that the $Q_1,\cdots,Q_s$ are monomial ideals by \cite[Proposition 3.8]{MY22}. Then this statement follows from \Cref{differential lemma} (2) and \Cref{scored primary diff powers}. 
\end{proof}

\begin{ex}
\label[ex]{normal diff powers}

Consider the normal semigroup $\N A$ from our running example \Cref{running ex 1}. Since $\phi_{F_1}=2x-y$ and $\phi_{0}=(2x-y)+(2y-x)=x+y$, by \Cref{scored diff powers prime} we have
\[
P_{F_1}^{\langle n\rangle}=\{x^ay^b \in \k[\N A]: 2a-b\ge n\}, \quad P_{F_0}^{\langle n\rangle}=\{x^ay^b\in \k[\N A]: a+b\ge n\}.
\]
\end{ex}

\begin{ex}
\label[ex]{scored diff powers ex}
We now consider the scored semigroup ring $\k[\N B]$ from \Cref{running ex 1}. Recall that $h_1(\eu)=h_1=2\eu_x-\eu_y$, $h_2(\eu)=h_2=2\eu_y-\eu_x$, $h_1(\N B)=\langle 2,3\rangle$ and $h_2(\N B)=\N$. We compute $\eta_\bb$ for a few $\bb\in \Z^d$ below.

\begin{table}[h!]
\renewcommand{\arraystretch}{1.25}
\centering
\begin{tabular}{|c|c|c|c|}
 \hline
 $\bb$ & $(h_1(\bb),h_2(\bb))$ & $\eta_\bb$ & $\deg(\eta_\bb)$ \\ 
 \hline
$(1,1)$ & $(1,1)$ & $h_1$ & $ 1 $  \\
$(-2,-2)$ & $(-2,-2)$ & $h_1(h_1-3)h_2(h_2-1)$ & 4 \\
$(1,0)$ & $(2,-1)$ & $ h_2$  & 1\\
 $(-1,1)$ & $(-3,3)$ & $h_1(h_1-2)(h_1-4) $ & 3 \\ 
 \hline
\end{tabular}
\label{scored semigroup table 1}
\end{table}

Since $h_2(\N B)=\N$, the form $h_2$ appears in our expression for $\eta_\bb$ if and only if $h_2(\bb)<0$. On the other hand, even if $h_1(\bb)>0$, $h_1(\bb)$ may be a gap of $\langle 2,3\rangle$ (see the first row) and give a factor of $h_1$ in our expression for $\eta_\bb$. 

Now consider the $P_{F_1}$-primary $J=\langle x^2y^2,x^3y^2,x^4y^2\rangle\subseteq \k[\N B]$ also defined in \Cref{running ex 1}. We can compute that $M_F(J)=\phi_{F_1}((2,1)+F)=3$. Consider  $x^\uu=x^2y^2$ and $x^\vv=x^4y^2$, so $\phi_{F_1}(\uu)=2$ and $\phi_{F_1}(\vv)=6$. Moreover,
\[
\lambda_J(\uu)=\deg(\eta_{(1,1)})=1, \quad \lambda_J(\vv)=\deg(\eta_{(-1,1)})=3=\phi_{F_1}(\vv)-M_{F_1}(J),
\]
where the orders of the canonical operators are calculated in the above table. Thus $x^\uu\in J^{\langle 1\rangle}\setminus J^{\langle 2\rangle}$ and $x^\vv\in J^{\langle 3\rangle}\setminus J^{\langle 4\rangle}$. This shows that the lower bound in \Cref{scored diff powers bounds} is sharp and that it is also possible for $\lambda_I(\uu)$ to lie strictly between our bounds. Our upper bound on $\lambda_I(\uu)$ is attained whenever $I$ is prime.

\end{ex}

\subsection{C-semigroups}

We now consider differential powers of monomial ideals in $\CC$-semigroups following the assumptions in \Cref{semigroup setup}.

\begin{defn}
Let $X\subseteq \k^d$ be a variety and $\uu\notin X$. The \emph{separation degree} of $\uu$ (over $X$) is
\[
\sepdeg(\uu,X)=\min\{\deg(g):g\in \I(X),\ g(\uu)\ne 0\}.
\]
We declare $\sepdeg(\uu,\emptyset)=0$ for all $\uu$. 
\end{defn}
The existence of some $g\in \I(X)$ such that $g(\uu)\ne 0$ follows from $\uu$ not lying in $X$, hence we may think of $g$ separating $\uu$ and $X$.

\begin{prop}
\label[prop]{C semigroup diff powers}
Let $\N A$ be a $\CC$-semigroup, and $I$ a $P_F$-primary ideal. Then $\lambda_I(\uu,\N A)=\lambda_I(\uu)$ is given below:
\[
\begin{cases*}
\displaystyle{\min_{\cc\in\std(I)} \sum_{i\in\supp(F)} \max\{h_i(\uu-\cc),0\}} & $\sqrt{I}\ne \mm$ \\
\displaystyle{\min_{\cc\in\std(I)} \left(\sum_{i\in\supp(F)} \max\{h_i(\uu-\cc),0\}+\sepdeg(\uu,\Theta(\cc-\uu))\right)} & $\sqrt{I}=\mm$.
\end{cases*}
\]
\end{prop}

\begin{proof}
Recall that $\I(\N A\setminus(-(\cc-\uu)+\N A))=\overline{\eta_{\cc-\uu}}(\eu)\I(\Theta(\cc-\uu))$. Furthermore, $\overline{\eta_{\cc-\uu}}(\uu)\ne 0$ because $\cc\in \N A$. Thus 
\[
\lambda_{I}(\uu,\N A)=\min_{\cc\in\std(I)}\{\deg(\overline{\eta_{\cc-\uu}}(\eu)g(\eu)):g(\eu)\in \I(\Theta(\cc-\uu),\ g(\uu)\ne 0\}
\]
If $\sqrt{I}=\mm$ then we are done by the definition of separation degree and $\overline{\eta_{\cc-\uu}}(\eu)$.

If $\sqrt{I}\ne \mm$, for all $\cc\in\std(I)$ we may choose $\xx\in F$ such that $h_i(\cc+\xx-\uu)=h_i(\cc-\uu)$ for all $i\in \supp(F)$ and $h_i(\cc+\xx-\uu)>0$ for all $i\notin \supp(F)$. We also assume that $h_i(\cc+\xx-\uu)>h_i(\ggg)$ for all $\ggg\in \gaps(\N A)$ and some $i\notin \supp(F)$. Since $h_i$ is nonnegative on $\N A$, for all $\bb\in \N A$ and $\ggg\in \gaps(\N A)$, $h_i(\bb+\cc+\xx-\uu)>h_i(\ggg)$. Thus $\bb+\cc+\xx-\uu\notin \gaps(\N A)$, $\Theta(\cc+\xx-\uu)=\emptyset$ and $\sepdeg(\uu,\Theta(\cc+\xx-\uu))=0$. In particular, $\deg(\overline{\eta_{\cc+\xx-\uu}}(\eu))\le \deg(\overline{\eta_{\cc-\uu}(\eu)})$, so we may compute $\lambda_I(\uu)$ by taking the following minimum:
\begin{align*}
\lambda_I(\uu)=\min_{\cc\in\std(I)}\deg(\eta_{\cc+\xx-\uu})&=\min_{\cc\in\std(I)} \sum_{i\in \supp(F)} \max\{h_i(\uu-\cc-\xx),0\}\\
&=\min_{\cc\in\std(I)} \sum_{i\in \supp(F)} \max\{h_i(\uu-\cc),0\}.
\end{align*}
\end{proof}

\begin{lemma}
\label[lemma]{theta lemma}
Let $\N A$ be a $\CC$-semigroup and $I\subseteq \k[\N A]$.
\begin{enumerate}
\item For every $\cc\in\std(I)$, and $\uu\in \N A$, $\sepdeg(\uu,\Theta(\cc-\uu))\le |\Theta(\cc-\uu)|\le |\gaps(\N A)|$.

\item For any $\cc\in\std(I)$ and $\uu$ such that $h_i(\uu-\cc)\ge 0$ for all $1\le i\le m$ and $\phi_0(\uu)$ sufficiently large, $\Theta(\cc-\uu)+\cc-\uu=\gaps(\N A)$, hence $|\Theta(\cc-\uu)|=|\gaps(\N A)|$.

\end{enumerate}
\end{lemma}

\begin{proof}
Given $\cc\in \std(I)$ and $\uu\in \N A$, 
\[
\Theta(\cc-\uu)+(\cc-\uu)=\{\ggg\in \gaps(\N A):\ggg+\uu-\cc\in \N A\}\subseteq \gaps(\N A).
\]
Taking the cardinality of these sets shows that $|\Theta(\cc-\uu)|\le |\gaps(\N A)|<\infty$. If we enumerate the (finite) set of points $\Theta(\cc-\uu)=\{\pp_1,\cdots,\pp_s\}$ then $\I(\Theta(\cc-\uu))=\I(\{\pp_1\})\cdots\I(\{\pp_s\})$ has a generating set consisting of $s$-fold products of linear forms. Thus $|\sepdeg(\uu,\Theta(\cc-\uu))|\le |\Theta(\cc-\uu)|$, proving our first point.

Since there are only finitely many values in $\phi_0(\gaps(\N A))$, we may choose $\uu$ such that $\phi_0(\uu-\cc)>\phi_0(\ggg)$ for all $\ggg\in \gaps(\N A)$. Since $\phi_0$ is positive on $\gaps(\N A)$, $\ggg+\uu-\cc\notin \gaps(\N A)$ for all $\ggg\in \gaps(\N A)$. Since $h_i(\uu-\cc)\ge 0$ for all $i$, $\uu-\cc\in \overline{\N A}$. Thus $\ggg+\uu-\cc\in \overline{\N A}\setminus \gaps(\N A)=\N A$ and we are done. 
\end{proof}

Combining the two points from the previous lemma shows that for any $\cc\in\std(I)$, if $\phi_0(\uu)$ is sufficiently large and $h_i(\uu-\cc)\ge 0$, then 
\[
\sepdeg(\uu,\Theta(\cc-\uu))=\sepdeg(\cc,\Theta(\cc-\uu)+\cc-\uu)=\sepdeg(\cc,\gaps(\N A).
\]
In particular, the separation degree is independent of $\uu$ (assuming $\phi_0(\uu)$ is sufficiently large).

\begin{cor}
\label[cor]{C semigroup diff powers prime}
Let $\k[\N A]$ be a $\CC$-semigroup. If $P_F\ne \mm$ then $P_F^{\langle n\rangle}=\langle x^\uu:\phi_F(\uu)\ge n\rangle$. On the other hand $\mm^{\langle n\rangle}=\langle x^\uu: \phi_0(\uu)+\sepdeg(\uu,\Theta(-\uu))\ge n\rangle$. 
\end{cor}

\begin{proof}
For any $\uu\in \N A$ and $\cc\in\std(P_F)=F$, $h_i(\uu-\cc)=h_i(\uu)\ge 0$. Thus $\max\{h_i(\uu-\cc),0\}=h_i(\uu)$ and by \Cref{C semigroup diff powers}, $\lambda_{P_F}(\uu)=\sum_{i\in \supp(F)}h_i(\uu)=\phi_F(\uu)$ and $\lambda_\mm(\uu)=\phi_0(\uu)+\sepdeg(\uu,\Theta(-\uu))$
\end{proof}

\begin{cor}
\label[cor]{C semigroup diff powers bounds}
Let $\k[\N A]$ be a $\CC$-semigroup ring and $I$ a $P_F$-primary ideal.
If $\sqrt{I}\ne \mm$ then $\phi_F(\uu)-M_F(I)\le \lambda_I(\uu)\le \phi_F(\uu)$. If $\sqrt{I}=\mm$ then $\lambda_I(\uu)\le \phi_0(\uu)+|\gaps(\N A)|$. 
\end{cor}

\begin{proof}
If $\sqrt{I}\ne \mm$, then for all $\cc\in F\subseteq \std(I)$, $h_i(\uu-\cc)=h_i(\uu)\ge 0$. Thus $\lambda_I(\uu,\N A) \le \sum_{i\in \supp(F)} h_i(\uu)=\phi_F(\uu)$. The upper bound for $\lambda_I(\uu)$ if $\sqrt{I}=\mm$ follows analogously.

For our lower bound, observe that for every $\cc\in\std(I)$,
\[
\sum_{i\in\supp(F)}\overbrace{\max\{h_i(\uu-\cc),0\}}^{\ge h_i(\uu-\cc)}\ge \phi_F(\uu-\cc)\ge \phi_F(\uu)-M_F(I).
\] 
Taking the minimum over all $\cc\in\std(I)$ gives the desired lower bound of $\lambda_I(\uu)$.
\end{proof}

\begin{cor}
\label[cor]{C diff powers}
Let $I$ be a monomial ideal of a $\CC$-semigroup with primary decomposition $I=Q_1\cap \cdots\cap Q_s$, where for each $1\le k\le s$, $\sqrt{Q_k}=P_{F_k}$ for some face $F_k$. Then,
\[
I^{\langle n\rangle}=\langle x^\uu:\lambda_{Q_k}(\uu)\ge n\text{ for all }1\le k\le s\rangle.
\]
\end{cor}

The proof of this statement is exactly the same as \Cref{scored diff powers}.

\begin{ex}
\label[ex]{C diff powers ex}
Consider the $\CC$-semigroup $\N C$ from \Cref{running ex 1} with $\gaps(\N C)=\{(1,2),(2,3),(2,4)\}$ and defining functionals $h_1=2\eu_x-\eu_y$ and $h_2=2\eu_y-\eu_x$. While we are guaranteed a generating set of $\I(\gaps(\N C))$ of degree at most three, we can actually find a generating set of quadratics
\[
\I(\gaps(\N C))=\langle (\eu_x-1)(\eu_x-2),(\eu_x-2)(\eu_y-2),\eu_y^2+2\eu_x-7\eu_y+8\rangle.
\]
Thus for any ideal $I\subseteq \k[\N C]$, $\cc\in\std(I)$ and $\uu\in \N C$, $\sepdeg(\uu,\Theta(\cc-\uu))\le 2$. 

We now compute the $\bb^{th}$-graded part of $\DD_{\k[\N C]}=\DD$ for various $\bb\in \Z^d$ in \Cref{C semigroup table 1}.

\begin{table}[ht]
\renewcommand{\arraystretch}{1.25}
\centering
\begin{tabular}{|c|c|c|}
\hline
$\bb$ & $(h_1(\bb),h_2(\bb))$ & $\DD_\bb$
\\
\hline
$(-1,-1)$ & $(-1,-1)$ & $x^{-1}y^{-1}h_1h_2\I(\{(3,4),(3,5)\})$ \\
$(-2,-2)$ & $(-2,-2)$ & $x^{-2}y^{-2}h_1(h_1-1)h_2(h_2-1)\I(\{(3,4),(4,5),(4,6)\})$ \\
$(-2,-4)$ & $(0,-6)$ & $x^{-2}y^{-4}h_2(h_2-1)\cdots(h_2-5)\I(\{(3,6),(4,8),(4,7)\})$ \\
$(-1,-4)$ & $(2,-7)$ & $x^{-1}y^{-4}h_2(h_2-1)\cdots(h_2-6)\k[\eu_x,\eu_y]$ \\
\hline
\end{tabular}
\caption{Graded pieces of $\DD_{\k[\N C]}$}
\label{C semigroup table 1}
\end{table}

We can compute the ideals in the first two rows to show that $\sepdeg((1,1),\Theta(-(1,1)))=1$ and $\sepdeg((2,2),\Theta(-(2,2)))=2$, hence $xy\in \mm^{\langle 3\rangle}\setminus \mm^{\langle 4\rangle}$ and $x^2y^2\in \mm^{\langle 6\rangle}\setminus \mm^{\langle 7\rangle}$.

Now consider the $\mm$-primary ideal $I\subseteq \k[\N C]$ where $\std(I)=\{(0,0),(1,1),(2,1)\}$ and fix $\uu=(3,5)$. We claim that finding the $\cc\in \std(I)$ that minimizes the sum $\sum_{i\in \supp(F)}\max\{h_i(\uu-\cc),0\}$ is not the same as finding the $\cc\in \std(I)$ that minimizes the sum in $\lambda_I(\uu)$. Indeed, the table below shows that $\sum_{i\in \supp(F)}\max\{h_i(\uu-\cc),0\}$ is minimized at $\cc=(1,1)$, while $\cc=(2,1)$ gives us the minimum order of some $\delta\in \DD_{\cc-\uu}$ where $\cc\in\std(I)$. The data for the last two rows were computed in \Cref{C semigroup table 1}.

\begin{center}
\renewcommand{\arraystretch}{1.25}

\begin{tabular}{|c|c|c|c|}
\hline
$\cc$ & $\cc-\uu$ & $\displaystyle{\sum_{i\in \supp(F)}\max\{h_i(\uu-\cc),0\}}$ & $\sepdeg(\uu,\Theta(\cc-\uu))$   \\
\hline
$(0,0)$ & $(-3,-5)$ & 8 & 2  \\
$(1,1)$ & $(-2,-4)$ & 6 & 2 \\
$(2,1)$ & $(-1,-4)$ & 7 & 0 \\
\hline
\end{tabular}
\end{center}

\end{ex}

\begin{rem}
\label[rem]{no initial degree rem}
In general, $\sepdeg(\uu,\Theta(\cc-\uu))$ need not equal the minimal degree of a nonzero polynomial in $\I(\Theta(\cc-\uu))$. Let $\N A=\N^2\setminus\{(1,0),(2,0),(5,0)\}$ and $\uu=(4,0)$. Then
\[
\Theta(-\uu)=\{(6,0),(9,0)\},\qquad
\I(\Theta(-\uu))=\langle\eu_y,(\eu_x-6)(\eu_x-9)\rangle.
\]
Every linear polynomial in this ideal vanishes at $\uu$, whereas evaluating $((\eu_x-6)(\eu_x-9))(\uu)=10\ne0$. Thus $\sepdeg(\uu,\Theta(-\uu))=2$.

Moreover, $\I(\gaps(\N A))=\langle\eu_y,(\eu_x-1)(\eu_x-2)(\eu_x-5)\rangle$, so $\sepdeg(\mathbf{0},\gaps(\N A))=3$. If $\uu=(u_1,u_2)\in\N A$ with $\phi_0(\uu)\ge5$, then either $u_2>0$ or $u_1\ge6$. Hence $\uu+\gaps(\N A)\subseteq\N A$, giving $\Theta(-\uu)=\uu+\gaps(\N A)$. Since translating by $\uu$ preserves separation degree, $
\sepdeg(\uu,\Theta(-\uu))
=\sepdeg(\mathbf{0},\gaps(\N A))=3$.

\end{rem}

\section{Applications of Differential Powers}

\subsection{The Containment Problem}

We now consider when the differential power topology $\{P^{\langle n\rangle}\}$ is cofinal with the $P$-adic topology. For the rest of this section we will assume that $\N A$ is a semigroup with generators $\aa_1,\cdots,\aa_r$, $D_F:=\max\{\phi_F(\aa_i):1\le i\le r\}$, and $D_0=D_{\{0\}}=\max\{\phi_0(\aa_i):1\le i\le r\}$. We also keep the assumptions from \Cref{semigroup setup}.

\begin{thm}
\label[thm]{:t prime containment}
Suppose \emph{either} $\N A$ is scored or $\N A$ is a $\CC$-semigroup and $P_F\ne \mm$. Then for any $n\ge 1$, 
\[
P_F^{\langle D_Fn\rangle}\subseteq P_F^{\langle D_Fn-D_F+1\rangle}\subseteq P_F^n.
\]
If $\N A$ is a $\CC$-semigroup and $\gamma:=|\gaps(\N A)|$, then 
\[
\mm^{\langle (D_0+\gamma)n\rangle}\subseteq \mm^{\langle (D_0+\gamma)(n-1)+1\rangle}\subseteq \mm^n.
\]
\end{thm}

\begin{proof}
Suppose $x^\uu\in P_F^{\langle D_Fn-D_F+1\rangle}$. Write $\uu=\sum_i c_i\aa_i=\sum_{\aa_i\in F}c_i\aa_i+\sum_{\aa_i\notin F}c_i \aa_i$. Since $P_F=\langle x^{\aa_i}:\aa_i\notin F\rangle$, if we let $t=\sum_{\aa_i\notin F}c_i$, then $x^\uu\in P_F^t$ and 
\[
\phi_F(\uu)=\sum_{\aa_i\notin F}c_i \phi_F(\aa_i)\le D_Ft. 
\]
But since $x^\uu\in P_F^{\langle D_Fn-D_F+1\rangle}$ we have $\phi_F(\uu)>D_F(n-1)$. Thus $D_Ft\ge \phi_F(\uu)>D_F(n-1)$, and cancelling $D_F$ from both sides gives us $t>n-1$. We conclude that $x^\uu\in P_F^t\subseteq P_F^n$ and we are done.

If $\N A$ is a $\CC$-semigroup, then $x^\uu\in \mm^{\langle (D_0+\gamma)(n-1)+1\rangle}$ implies that $\phi_0(\uu)+\sepdeg(\uu,\Theta(-\uu))>(D_0+\gamma)(n-1)$. But $\sepdeg(\uu,\Theta(-\uu))\le \gamma\le \gamma(n-1)$ for $n\ge 2$, and our theorem is trivial if $n=1$. Thus $\phi_0(\uu)>D_0(n-1)$. The rest of the proof follows analogously to the previous case.
\end{proof}

\begin{rem}
If $\N A$ is a $\CC$-semigroup and we are only concerned with finding some $C>0$ such that $\mm^{\langle Cn\rangle}\subseteq \mm^n$ for all but finitely many $n$ (which still shows that the $\{\mm^{\langle n\rangle}\}_{n\ge 0}\sim \{\mm^n\}_{n\ge 0}$) we may choose a smaller constant than in \Cref{:t prime containment}. Indeed, if $n> |\gaps(\N A)|$ then $\mm^{\langle (D_0+1)(n-1)+1\rangle}\subseteq \mm^n$, where the argument follows the same logic. 
\end{rem}

\begin{ex}
Recall $\k[\N A]$ from \Cref{running ex 1}, the normal semigroup with defining functionals $h_1=2x-y$ and $h_2=2y-x$. If $F_1=\{(a,b):h_1(a,b)=0\}$ then $D_{F_1}=h_1(2,1)=3$. Thus $P_{F_1}^{\langle 3n\rangle}\subseteq P_{F_1}^n$ for all $n\ge 0$. Moreover, one can show that 
\[
P_{F_1}^{\langle 6\rangle}=\langle x^4y^2\rangle\subseteq P_{F_1}^2,\qquad P_{F_1}^{\langle 6\rangle}\not\subseteq P_{F_1}^3.
\]
Thus our value $D_F=3$ for which $P^{\langle nD_F\rangle}\subseteq P^n$ is sharp. A similar containment happens when $n$ is any multiple of 6, so this is not behavior that only happens for a few exponents.

We now show that while $\mm^{\langle (D+1)(n-1)+1\rangle}\subseteq \mm^n$ for all $n$ sufficiently large, this is not true for small $n$. Consider the $\CC$-semigroup $\N A=\N^2\setminus \{(1,0),(2,0),(5,0)\}=\displaystyle{A=\m{0 & 1 & 2 & 3 & 4 \\1 & 1 & 1 & 0 & 0}}$. Then for all $n$ sufficiently large, $\mm^{\langle (D+1)(n-1)+1\rangle}=\mm^{\langle 5n-4\rangle}\subseteq \mm^n$. However, we claim that this is not true for $n=2$, that is $\mm^{\langle 6\rangle}\not\subseteq \mm^2$. Consider $x^4\in \mm\setminus \mm^2$. We saw in \Cref{no initial degree rem} that $\sepdeg((4,0),\Theta((-4,0)))=2$ so $\lambda_\mm(\uu)=4+2=6$. Thus $x^\uu\in \mm^{\langle 6\rangle}$ but not in $\mm^2$. 
\end{ex}

We also remark that the containment $P_F^{(D_Fn)}\subseteq P_{F}^n$ was shown in \cite{Wal18}, showing that in this case the symbolic and differential powers behave similarly. 

\begin{thm}
\label[thm]{cofinality theorem}
Suppose $\N A$ be a scored or $\CC$-semigroup and $I\subseteq \k[\N A]$ a monomial ideal. Then there exists some $C(I)=C>0$ such that 
\[
I^{Cn}\subseteq I^{(Cn)}\subseteq I^{\langle Cn\rangle}\subseteq I^n.
\]
In particular, the topologies defined by ordinary powers, symbolic powers, and differential powers of $I$ are all cofinal.
\end{thm}

\begin{proof}
The first two containments follow from \Cref{symbolic lemma}. To show the last containment, we use the existence of some $k>0$ such that for any fixed $n\ge 0$, we may choose a primary decomposition of $I^n$, $I^n=\bigcap_{i=1}^{s} Q_{i}$ such that $P_{i}^{kn}\subseteq Q_{i}$ where $P_{i}=\sqrt{Q_{i}}$, \cite[Theorem 3.4]{Swa97}. Since each $P_{i}\in \ass(I^n)$ and $I^n$ is a monomial ideal, we are guaranteed that $P_{i}$ is a \emph{monomial} prime ideal for all $i$, though the $Q_{i}$ may not be monomial ideals. 

Then for each $n$, $I\subseteq \bigcap_{i=1}^{s} P_{i}$ and by \Cref{:t prime containment} we may choose $C'$ such that $P^{\langle C'n\rangle}\subseteq P^n$ for all monomial primes $P$. Then,
\[
I^{\langle C'kn\rangle}\subseteq \bigcap_{i=1}^s P_i^{\langle C'kn\rangle}\subseteq \bigcap_{i=1}^s P_i^{kn}\subseteq \bigcap_{i=1}^s Q_i=I^n.
\]
\end{proof}

We also remark that even if $\k[\N A]$ is a polynomial ring, there does not exist a constant $C$ independent of $I$, \cite[Proposition 5.3]{KMP+26}. 

\subsection{Convergence of Differential Signature}

We now use our description of differential powers to show that differential signature exists for scored semigroups and $\CC$-semigroups.

\begin{defn}
Let $(R,\mm)$ be a local ring and $\k=R/\mm$. Then the \emph{differential signature} of $R$ is
\[
s_\k^{\diff}(R)=s^{\diff}(R)=\limsup_{n\rightarrow \infty} \frac{\ell_R(R/\mm^{\langle n+1\rangle_\k})}{n^d/d!}.
\]
\end{defn}

While the affine semigroup rings in this article are non-local, taking differential powers commutes with localization and completion. Thus when finding differential signature we may take the completion or localization of our semigroup ring at the maximal ideal without issue, \cite[Lemma 3.10, Lemma 3.11]{BJNB19}.

\begin{prop}
\label[prop]{scored graded diff powers}
If \emph{either} $\N A$ is a scored semigroup and $I$ a radical ideal, \emph{or} $\N A$ is a $\CC$-semigroup and $I$ a radical ideal such that $\mm\notin \ass(I)$, then $\{I^{\langle n\rangle}\}_{n\ge 0}$ is a graded family, i.e., $I^{\langle n\rangle}I^{\langle n'\rangle}\subseteq I^{\langle n+n'\rangle}$.
\end{prop}

\begin{proof}
Since associated primes of $I$ are also monomial ideals, there exist faces $F_1,\cdots,F_s$ such that $\displaystyle{I=P_{F_1}\cap \cdots P_{F_s}}$. Since taking differential powers commutes with intersections, for each $n\ge 0$ we have 
\[
I^{\langle n\rangle}=\bigcap_{i=1}^s P_{F_i}^{\langle n\rangle}=\langle x^\uu: \phi_{F_i}(\uu)\ge n \text{ for all } 1\le i\le s\rangle.
\]
The statement now follows by linearity of the $\phi_{F_i}$.
\end{proof}

In particular, if $\mm$ is the maximal ideal of a scored semigroup ring $R=\k[\![\N A]\!]$, then $\{ \mm^{\langle n\rangle}\}_{n\ge 0}$ is graded and hence $s^{\diff}(R)$ exists by \cite[Theorem 1.1]{Cut14}.

\begin{ex}
It is not true in general that the maximal ideal of a $\CC$-semigroup is graded. Consider $\k[\N C]$ from \Cref{running ex 1}. The first two rows of \Cref{C semigroup table 1} shows that $xy\in \mm^{\langle 3\rangle}$ and $x^2y^2\in \mm^{\langle 6\rangle}$. However, one can show that $\lambda_\mm((3,3))=8$ and hence $(xy)(x^2y^2)=x^3y^3\notin \mm^{\langle 9\rangle}$.
\end{ex}

The following proposition generalizes the proof of \cite[Theorem 7.4]{BJNB19}

\begin{prop}
\label[prop]{diff sig}
Let $R=\k[\![\N A]\!]$ where $\N A$ is either a scored or $\CC$-semigroup. Then 
\[
s^{\diff}(R)=s^{\diff}(\overline{R})=d!\ \text{vol}(\PP),
\]
where $\PP$ is the convex polytope $\PP=\{\uu\in \R_{\ge 0} A:\phi_0(\uu)\le 1\}$.
\end{prop}

\begin{proof}
Let $\mm$ and $\nn$ denote the maximal ideals of $R$ and $\overline{R}$ respectively. By Ehrhart's theorem and \Cref{scored diff powers prime}, as $n\rightarrow \infty$,
\[
\frac{\ell_{\overline{R}}(\overline{R}/\nn^{\langle n+1\rangle})}{n^d}=\frac{|\{\uu\in \overline{\N A}:\phi_0(\uu)\le n\}|}{n^d}=\frac{|n\PP\cap \Z^d|}{n^d}\longrightarrow \text{vol}(\PP).
\]
If $\N A$ is scored, then $\overline{\N A}\setminus \N A$ consists of finitely many hyperplane sections of $\overline{\N A}$. There are $O(n^{d-1})$ lattice points in these hyperplane sections that also lie in $n\PP$. Thus 
\[
\ell_R(R/\mm^{\langle n+1\rangle})=\ell_{\overline{R}}(\overline{R}/\nn^{\langle n+1\rangle})-O(n^{d-1})=|n\PP\cap \Z^d|-O(n^{d-1}).
\]
Dividing both sides by $n^d$ and taking the limit gives us the desired result.

Now suppose $R$ is a $\CC$-semigroup. Since $0\le \sepdeg(\uu,\Theta(-\uu))\le |\gaps(\N A)|$, 
\[
J_n:=\langle x^\uu:\phi_0(\uu)+1> n\rangle\subseteq \mm^{\langle n+1\rangle}\subseteq \langle x^\uu:\phi_0(\uu)+|\gaps(\N A)|>n\rangle=:J_n'.
\]
Moreover, $J_n$ and $\nn^{\langle n\rangle}$ differ by finitely many monomials: those whose exponents lie in $\gaps(\N A)$ and satisfy $\phi_0(\uu)+1>n$. Similarly $J_n'$ and $\nn^{\langle n-|\gaps(\N A)|+1\rangle}$ differ by finitely many monomials. Thus
\[
\limsup_{n\rightarrow \infty} \frac{\ell_{\overline{R}}(\overline{R}/\nn^{\langle n-|\gaps(\N A)|+1\rangle})}{n^d}\le \limsup_{n\rightarrow \infty} \frac{\ell_{{R}}({R}/\mm^{\langle n+1\rangle})}{n^d}\le \limsup_{n\rightarrow \infty} \frac{\ell_{\overline{R}}(\overline{R}/\nn^{\langle n\rangle})}{n^d}.
\]
The left and right hand sides both converge to $\text{vol}(\PP)$, so our proof is complete.
\end{proof}

So even though $\mm^{\langle n\rangle}$ is not a graded family if $R$ is a $\CC$-semigroup ring, we still have existence of differential signature.

\subsection{Pinched Veronese rings}

We conclude this article by studying the class of pinched Veronese rings. These are obtained by taking a Veronese ring and removing a single generator. We will show that in many cases we obtain either a scored or a $\CC$-semigroup ring and be able to apply our previous results here.

Given integers $a>1$ and $d\ge 1$, the Veronese ring (of degree $a$ in dimension $d$) is the affine semigroup ring over the semigroup $V_{d,a}=\langle\vv\in \N^d: |\vv|=a\rangle$. However, since $\Z V_{d,a}\ne \Z^d$ (for example $\ee_1\notin \Z V_{d,a}$), this is not the setting in which we want to study Veronese rings. Instead, we will construct a semigroup $\N A$ such that $\Z A=\Z^d$ and $\N A\simeq V_{d,a}$.

Let $\N A_{d,a}=\bigcap_{i=1}^d \{\uu\in\N^d:h_i(\uu)\ge 0\}$ where
\[
h_1=ax_1-x_2-\cdots-x_d,\quad h_2=x_2,\quad \ldots \quad h_d=x_d.
\]
Consider the linear map $\psi:\N A_{d,a}\rightarrow \N^d$ defined by
\begin{equation}
\label{veronese iso}
\uu\mapsto (h_1(\uu),\cdots,h_d(\uu)).
\end{equation}

We claim that $\psi:\N A_{d,a}\rightarrow V_{d,a}$ is our desired isomorphism of semigroups. First, because $|\psi(\uu)|=h_1(\uu)+\cdots+h_d(\uu)=au_1$, $\psi(\N A_{d,a})\subseteq V_{d,a}$. The map $\psi$ is injective because $\psi$ can be represented by a $d\times d$ matrix with linearly independent columns. To see it is bijective, observe that $\psi^{-1}(\vv)=\left(\frac{|\vv|}{a},v_2,\cdots,v_d\right)$ is indeed a well-defined inverse of $\psi$. Applying $\psi^{-1}$ to a generating set of $V_{d,a}$ shows that $\{(1,v_2,\cdots,v_d):v_2+\cdots+v_d\le a\}$ is a generating set for $\N A_{d,a}$.

\begin{cor}
If $\k[V_{d,a}]$ is a $d$-dimensional Veronese ring in degree $a$ and $P$ a monomial prime of $\k[V_{d,a}]$ then
\[
P^{(an)}\subseteq P^{\langle an\rangle}\subseteq P^n.
\]
\end{cor}

\begin{proof}
It suffices to apply \Cref{:t prime containment} to some monomial prime in $\N A_{d,a}$. Since $\phi_0(1,v_2,\cdots,v_d)=a$ for all $v_2,\cdots,v_d\in \N$ we are done.
\end{proof}

We now consider the pinched Veronese semigroup $V_{d,a,\bb}=\langle\vv\in \N^d: |\vv|=a,\vv\ne \bb\rangle$. The following statement discusses the structure of $\N A_{d,a,\bb}:=\psi^{-1}(V_{d,a,\bb})$ which depends on the value of $\max(\bb):=\max_i\{b_i\}$.

\begin{prop}
Let $\psi$ be defined as in \eqref{veronese iso} and $V_{d,a,\bb}$ be a pinched Veronese semigroup. 

\begin{enumerate}
\item If $\max(\bb)<a-1$ then $\N A_{d,a,\bb}$ is a $\CC$-semigroup.

\item If $\max(\bb)=a-1$ and $a>2$ then $\N A_{d,a}\setminus \N A_{d,a,\bb}$ consists of all lattice points on a single ray parallel to an extremal ray of $\R_{\ge 0} A_{d,a}$. 

\item If $\max(\bb)=a-1$, $a=2$, and $d=2$ (i.e., $\k[V_{d,a,\bb}]=\k[x^2,y^2]$) then $\N A_{d,a,\bb}$ is isomorphic to $\N^2$. If $d>2$ then $V_{d,a}\setminus V_{d,a,\bb}$ consists of infinitely many rays parallel to an extremal ray of $\R_{\ge 0} A_{d,a}$.

\item If $\max(\bb)=a$ and $d=2$ then $\N A_{d,a,\bb}\simeq \N A_{d,a-1}$. If $d>2$ then 
\[
\N A_{d,a,\bb}=\{\uu \in \N A_{d,a}:u_1\le u_2+\cdots+u_d\}.
\]
\end{enumerate}
\end{prop}

\begin{proof}
The first point follows from \cite[Lemma 2.17]{MP23}, and the second and third point follow from \cite[Lemma 2.18]{MP23}. 

If $\max(\bb)=a$ and $d=2$, we may assume $\bb=a\ee_2$. Since $\psi^{-1}(a\ee_2)=(1,a)$, $\N A_{d,a,\bb}$ is generated by $\{(1,0),\cdots,(1,a-1)\}$ which is exactly a generating set for $\N A_{d,a-1}$. 

Now suppose $\max(\bb)=a$, $d>2$ and assume without loss of generality that $\bb=a \ee_1$. Then $\psi^{-1}(a\ee_1)=\ee_1$ and $\N A_{d,a,a \ee_1}=\langle \{(1,u_2,\cdots,u_d):1\le u_2+\cdots+u_d\le a\}\rangle$. That is,
\[
\N A_{d,a}\setminus \N A_{d,a,a \ee_1}=\{\uu\in \N A_{d,a}:u_1>u_2+\cdots+u_d\}.
\]
\end{proof}

When $\max(\bb)=a$ and $d>2$, $\N A_{d,a,\bb}$ is isomorphic to the normal semigroup with defining functionals
\[
h_1=ax_1-x_2-\cdots-x_d,\quad h_2=x_2,\quad \ldots, \quad h_d=x_d,\quad h_{d+1}=x_2+\cdots+x_d-x_1.
\]
If $\max(\bb)=a$ and $d=2$ then $\N A_{d,a,\bb}$ is isomorphic to the normal semigroup with defining functionals $h_1=(a-1)x_1-x_2$, $h_2=x_2$.

The following corollary is a direct application of \Cref{:t prime containment}. 

\begin{cor}
\label[cor]{pinched veronese containment}
Let $\k[V_{d,a,\bb}]$ be a pinched Veronese ring and $P\subseteq \k[V_{d,a,\bb}]$ a monomial prime ideal. Consider the constant $D$ defined by the following conditions 
\begin{enumerate}
\item If $\max(\bb)<a-1$ and $P$ is non-maximal, let $D=a$. If $P$ is maximal let $D=a+1$. 

\item If $\max(\bb)=a-1$, $d=2$ and $a>2$, let $D=a$. 

\item If $\max(\bb)=a$ and $d=2$ let $D=a-1$. If $d>2$ let $D=2a-1$.
\end{enumerate}

If $V_{d,a,\bb}$ and $P$ satisfy any of the above three conditions, then 
\[
P^{(Dn)}\subseteq P^{\langle Dn\rangle}\subseteq P^n.
\]
\end{cor}

We also apply \Cref{cofinality theorem} and \Cref{diff sig} to obtain the following.

\begin{cor}
Suppose $\k[V_{d,a,\bb}]$ is a pinched Veronese ring satisfying any of the three conditions in the statement of \Cref{pinched veronese containment}, and $D$ is chosen as above.
\begin{enumerate}
\item If $I\subseteq \k[V_{d,a,\bb}]$ is a monomial ideal then there exists some $C(I)=C>0$ such that
\[
I^{Cn}\subseteq I^{(Cn)}\subseteq I^{\langle Cn\rangle}\subseteq I^n.
\]
In particular, the topologies defined by $\{I^n\}_{n\ge 0}, \{I^{(n)}\}_{n\ge 0}$ and $\{I^{\langle n\rangle}\}_{n\ge 0}$ are equivalent.

\item The differential signature $s^{\diff}(\k[V_{d,a,\bb}])$ exists and is equal to the $s^{\diff}(\k[\overline{V_{d,a,\bb}}])$. 
\end{enumerate}
\end{cor}

We may also compute the differential signature for the pinched Veronese rings satisfying any of the three conditions in \Cref{pinched veronese containment} via interated integration. When $\overline{V_{d,a,\bb}}=V_{d,a}$, we may induce on $d$ to show that $s^{\diff}(\k[V_{d,a,\bb}])=\frac{1}{a}$. If $\max(\bb)=a$ and $d>2$, the expression for the volume is more complicated.

\newpage

\bibliographystyle{amsalpha}
\bibliography{References}

\end{document}